\documentclass[11pt,letterpaper]{amsart}

\usepackage[utf8]{inputenc}
\usepackage{amssymb, amscd, amsmath, epsfig, mathtools}
\usepackage{amsthm}
\usepackage{enumerate}
\usepackage{xcolor}
\usepackage{scalerel}
\usepackage{soul}
\usepackage{tikz-cd}
\usepackage{esint}
\usepackage{cancel}
\usepackage{hyperref}
\usepackage{slashed}
\usepackage{url}

\usepackage[margin=1.0in]{geometry}

\newtheorem{theorem}{Theorem}[section]
\newtheorem*{theorem*}{Theorem}

\newtheorem{corollary}{Corollary}[section]
\newtheorem*{corollary*}{Corollary}
\newtheorem{lemma}{Lemma}[section]
\newtheorem{proposition}{Proposition}[section]
\theoremstyle{definition}
\newtheorem{remark}{Remark}[section]

\newcommand{\R}{\mathbb R}

\begin{document}

\title[Positivity from $ X \times \mathbb{S}^{1} $ to $ X $]{The Rosenberg $ \mathbb{S}^{1} $-Stability Conjecture for $ \chi(X) = 0 $}
\author{Jie Xu}
\address{Department of Technology, Operations and Statistics, New York University, New York, NY, USA, 10012}
\email{jx479@nyu.edu}
%\date{April 2025}

\begin{abstract}
Let $ X $ be a closed, oriented manifold with $ \dim X \geqslant 5 $. In this article, we show that 2006 Rosenberg's $ \mathbb{S}^{1} $-stability holds when $ X $ has zero Euler characteristic. The 2006 Rosenberg-Stolz Conjecture for $ X \times \R $ also follows under the same assumption, provided that the Riemannian metric $ g $ on $ X \times \R $ is complete, is of bounded curvature, and whose smallest eigenvalue is uniformly bounded below by some positive constant. We then show a $ \mathbb{T}^{n} $-stability theorem with the same hypothesis of $ X $.
\end{abstract}
\maketitle

\section{Introduction}
Let $ X $ be a closed, oriented manifold, $ \dim X \geqslant 5 $. We denote by $ \chi(X) $ the Euler characteristic of $ X $. We also denote by $ g $ the Riemannain metric on some manifold, and $ R_{g} $ be the associated scalar curvature. We say that $ g $ is a positive scalar curvature (PSC) metric if $ R_{g} > 0 $ everywhere. We say that $ g $ is a uniformly positive scalar curvature metric if there exists some positive constant $ \kappa > 0 $ such that $ R_{g} \geqslant \kappa > 0 $ everywhere.

A 2006 Rosenberg $ \mathbb{S}^{1} $-Stability Conjecture \cite{JR} asked whether the following statement is true in the sense of ``transpose" of the positivity: \\
\medskip

\noindent {\it{$ X \times \mathbb{S}^{1} $ admits a positive scalar curvature metric if and only if $ X $ admits a positive scalar curvature metric.}}
\medskip

One direction is trivial:  $X\times \mathbb{S}^1$ has a product PSC metric if $X$ has a PSC metric. Since $ \mathbb{S}^{1} $ is the one-point compactification of $ \mathbb{R} $, it is closely related to a 1994 Rosenberg-Stolz Conjecture \cite{RosSto} in the contrapositive statement:
\medskip

\noindent {\it{$ X \times \R $ admits a complete, positive scalar curvature metric if and only if $ X $ admits a positive scalar curvature metric.}}
\medskip

The two conjectures are related, in that a PSC metric $g$ on $X \times \mathbb{S}^{1}$ lifts to the PSC metric $ \pi^*g $ on $X \times \R$, where  
$ \pi : X \times \R \rightarrow X \times \mathbb{S}^{1} $ is the smooth covering map.  Note that $\pi^*g$  
is of  {\it{bounded curvature}} in the sense of of Aubin \cite[Chapter 2]{Aubin}: $ g $ is complete, has positive injectivity radius, and its sectional curvatures are uniformly bounded. In addition, the smallest eigenvalue of  $ \pi^{*}g $ is uniformly bounded below by a positive constant.

The main technical result answers the Rosenberg-Stolz conjecture for a wide class of metrics on manifolds $ X \times \mathbb{S}^1$ and immediately implies the 
 $ \mathbb{S}^{1} $-stability conjecture, provided $\chi(X) =0$ and $ \dim X \geqslant 5 $. In particular, the $\mathbb{S}^1$-stability conjecture holds if $\dim(X)= 2k - 1, k \geqslant 3 $.
\begin{theorem}\label{intro:thm1}
 (i)   Let $ X $ be a closed, oriented manifold with $ \dim X \geqslant 5 $. Assume that $ (X \times \R, g) $ admits a complete, uniformly PSC metric $ g $ of bounded curvature. Assume also that the smallest eigenvalue of $ g $ is uniformly bounded below by some $ \lambda > 0 $. If $ \chi(X) = 0 $, then $ X $ admits a PSC metric.

 (ii) Let $ X $ be a closed, oriented manifold with $ \dim X \geqslant 5 $ and $\chi(X) =0.$ Then 
 $ X \times \mathbb{S}^{1} $ admits a  PSC metric if and only if $ X $ admits a PSC metric.
\end{theorem}

With the same hypothesis on $ X $, we generalize the $ \mathbb{S}^{1} $-stability conjecture to a $ \mathbb{T}^{n} $-stability theorem:
\begin{theorem}\label{intro:thm3}
 Let $ X $ be a closed, oriented manifold with $ \dim X \geqslant 5 $. If $ \chi(X) = 0 $, then  $ X $ admits a PSC metric if and only if $ X \times \mathbb{T}^{n}, n \geqslant 1 $.
\end{theorem}
We treat this type of problem in \cite{RX2}, \cite{XU11}, \cite{Xu12}, \cite{XU10} by imposing a geometric condition---a conformally invariant angle condition, which we now recall for $ X \times \R $. 

Define two natural projections
\begin{equation*}
\pi_{X} : X \times \R \rightarrow X, \pi_{\R} : X \times \R \rightarrow \R.
\end{equation*}
If we denote by $ \xi $ the global coordinate along the $ \R $-direction, we mean that the value $ \xi = \pi_{\R}(p), p \in X \times \R $ is the global coordinate of the point $ p $ in $ \R $-direction. Consequently, we denote by $ \partial_{\xi} $ the global vector field on $ X \times \R $.

For the hypersurface $ X \times \{ P\} \subset X \times \mathbb{R} $, where $ P\in \R $ is fixed, we want to apply the Gauss-Codazzi equation
\begin{equation*}
R_{\imath_{P}^{*}g} = R_{g} - 2\text{Ric}_{g}(\nu_{g}, \nu_{g}) + h_{g}^{2} - \lvert A_{g} \rvert^{2} \; {\rm on} \; X \times \{ P \}
\end{equation*}
to detect the positivity of $ R_{\imath_{P}^{*}g} $. Here $ g $ is the complete Riemannian metric on the ambient space $ X \times \mathbb{R} $, $ \imath_{P} : X \times \{P\} \rightarrow X \times \mathbb{R} $ is the natural inclusion, $ A_{g} $ is the second fundamental form on the hypersurface $ X \times \{ P \} $, $ \nu_{g} $ is the unit normal vector field along $ X \times \{ P \} $. $ \text{Ric}_{g} $ is the Ricci curvature tensor, and $ h_{g} $ is the mean curvature on $ X \times \{ P \} $. We can choose $ \nu_{g} $ such that the $ g $-angle between $ \nu_{g} $ and $ \partial_{\xi} $ is acute. We showed in \cite{XU11} that
\begin{theorem}\cite[Theorem 1.1]{XU11}\label{intro:thm2}
Let $ X $ be an oriented, closed manifold. Let $ (M = X \times \R, g) $ be a Riemannian noncompact cylinder of bounded curvature such that $ n : = \dim M \geqslant 3 $. If $ g $ is a uniformly PSC metric, and  \begin{equation}\label{intro:eqn1}
    \angle_{g}(\nu_{g}, \partial_{\xi}) : = \cos^{-1} \left( \frac{g(\partial_{\xi}, \nu_{g})}{g(\nu_{g}, \nu_{g})^{\frac{1}{2}}g(\partial_{\xi}, \partial_{\xi})^{\frac{1}{2}}} \right) = \cos^{-1} \left( \frac{g(\partial_{\xi}, \nu_{g})}{g(\partial_{\xi}, \partial_{\xi})^{\frac{1}{2}}} \right) \in [0, \frac{\pi}{4})
\end{equation}
along the hypersurface $ \{ \xi = P \} $ for some $ P \in \R $, then there exists a complete metric $ g' $ in the conformal class of $ g $ such that $ \imath_{P}^{*} g' $ is a PSC metric on $ \{ \xi = P \} $.
\end{theorem}
\begin{remark}\label{intro:re1}
Since we can identify $ \{ \xi = P \}  = X \times \{P \} \cong X $, Theorem \ref{intro:thm2} gives a PSC metric $ \imath_{P}^{*} g' $ on $ X $. 

Suppose that there exists a global diffeomorphism $ F: X \times \R \rightarrow X \times \R $ whose global inverse is denoted by $ G $. Denote by $ \tilde{\xi} = \pi_{\R} \circ G $ as a new global coordinate along $ \R $-direction.  Then Theorem \ref{intro:thm2} gives a PSC metric on the hypersurface $ \{ \tilde{\xi} = 0 \} $, provided that (\ref{intro:eqn1}) is replaced by $ \angle_{g}(\nu_{g, \tilde{\xi}}, \partial_{\tilde{\xi}}) \in [0, \frac{\pi}{4}) $. Here $ \nu_{g, \tilde{\xi}} $ is the unit normal vector field along $ \{ \tilde{\xi} = 0 \} $. In general, we cannot identify $ \{ \tilde{\xi} = 0 \} $ with $ X $.
\end{remark}

Our major contribution is to show that
\begin{align}\label{intro:eqn2}
   & \chi(X) = 0  \\
   \Rightarrow & \exists F : X \times \R \xrightarrow{\cong} X \times \R, \tilde{\xi} : = \pi_{\R} \circ G, \; \rm{s.t.} \; \angle_{g}(\nu_{g, \tilde{\xi}}, \partial_{\tilde{\xi}}) \in [0, \frac{\pi}{4}) \; {\rm on} \; \{ \tilde{\xi} = 0 \} \nonumber,
\end{align}
i.e. $ \chi(X) = 0 $ implies that the angle condition (\ref{intro:eqn1}) holds with respect to $ \angle_{g}(\nu_{\tilde{g}}, \partial_{\tilde{\xi}}) $. We use that the vanishing of the Euler characteristic is equivalent to the existence of a nowhere vanishing vector field on $X$. This is in the same manner as in the fundamental work of Markus \cite{LM}, where he showed that $ \chi(X) = 0 $ implies the existence of the Lorentz metric.

Classically, the Rosenberg-Stolz conjectures were verified by imposing some topological conditions: either when $ X $ has low dimensions, see \cite{CL} for $ \dim X = 2 $, \cite{Chodosh} for $ \dim X = 3 $ and \cite{Rade} for $ \dim X = 5, 6 $; however there are counterexamples when $ \dim X = 4 $ due to Seiberg-Witten theory, see e.g.,\ \cite{JR}; or $ X $ is enlargeable \cite{GromovL}; or $ X $ is spin and satisfies some Rosenberg index conditions \cite{Zeidler}. In addition, there are also other topological restrictions when $ X $ is spin, see e.g.,\ \cite{CRZ}, \cite{CZ}, \cite{RosSto}, \cite{Zeidler2}, etc. The low dimensional scenarios are often solved by geometric approaches, from classical minimal hypersurfaces \cite{SY} to $ \mu $-bubble method developed by Gromov \cite{GROMOV}, \cite{GROMOV2}. The dimensional restriction is is forced by problems with the existence of smooth minimal hypersurfaces. We may expect a higher dimensional threshold, say $ \dim X \leqslant 10 $ by very recent work \cite{CMSW}.

Our topological condition (\ref{intro:eqn2}) is more general and dimensional free. In addition, (\ref{intro:eqn2}) builds the relation between a fundamental topological condition and a geometric condition in terms of Riemannian metrics that does not depend on curvature tensors or any derivatives of the Riemannian metrics.

This article is organized as follows: In \S2 we first convert our angle condition (\ref{intro:eqn1}) to an equivalent statement in terms of $ g_{\xi\xi} $ and $ g^{\xi \xi} $. Then we verify (\ref{intro:eqn2}) in Theorem \ref{Angle:thm1}, where the hypothesis $ \chi(X) = 0 $ is applied. (\ref{intro:eqn2}) is equivalent to (\ref{Angle:eqn8}) for later use. In \S3, we conclude that when $ \chi(X) = 0 $ and $ \dim X \geqslant 5 $, the Rosenberg-Stolz conjecture for $ X \times \R $ holds (Theorem \ref{RS:thm1}) under the hypotheses of Theorem 1.1(i). We then conclude that the $ \mathbb{S}^{1} $-stability conjecture holds (Corollary \ref{RS:cor1}) for all $ X $ with $ \dim X \geqslant 5 $, provided that $ \chi(X) = 0 $. Finally, under the same hypothesis on $ X $, we show that $ X \times \mathbb{T}^{n} $ admits a PSC metric if and only if $ X $ admits a PSC metric (Corollary \ref{RS:cor2}).
\medskip

{\bf{Acknowledgement}}: The author would like to thank Professor LeBrun for counterexamples in 4-manifolds with $ \chi(X) = 0 $. The author would also like to thank Professor Rosenberg for many discussions.
\medskip

\section{Reworking  of the Angle Condition and the Verification of (\ref{intro:eqn2})}
In this section, we show that the topological hypothesis $ \chi(X) = 0 $ connects with the geometric angle condition (\ref{intro:eqn1}) by showing that there exists a Riemannian metric $ g $ on $ X \times \R $ such that (\ref{intro:eqn2}) holds. This is used in \S3 to our results on the Rosenberg $ \mathbb{S}^{1} $-stability conjecture and Rosenberg-Stolz conjecture for $ X \times \R $.

Without loss of generality, we set $ P = 0 \in \R $. Geometrically, the angle condition (\ref{intro:eqn1}) says that on every tangent space $ T_{y}(X \times \R), y \in X \times \{ 0 \} $, the angle between $ \partial_{\xi} $ and $ \nu_{g} $ is not too large in terms of the $ g $-inner product. This is highly related to, but essentially not quite the same as expecting the metric $ g $ to be not ``far away" from the product type. 

We simplify this angle condition first. The angle condition (\ref{intro:eqn1}) is equivalent to
\begin{equation}\label{Angle:eqn1}
\cos^{-1} \left( \frac{g(\partial_{\xi}, \nu_{g})}{g(\partial_{\xi}, \partial_{\xi})^{\frac{1}{2}}} \right) \in [0, \frac{\pi}{4}) \Leftrightarrow \frac{g(\partial_{\xi}, \nu_{g})}{g(\partial_{\xi}, \partial_{\xi})^{\frac{1}{2}}} \in (\frac{\sqrt{2}}{2}, 1] \; {\rm along} \; X_{0}.
\end{equation}
Clearly $ X_{0} = \lbrace \xi = 0 \rbrace $. For any choice of local coordinates $ (x^{1}, \dotso, x^{n}, \xi) $ around some point $ y \in X_{0} $, the unit normal vector field is of the form
\begin{align*}
    \nu_{g} & = \frac{\text{grad}_{g}\xi}{\sqrt{g(\text{grad}_{g}\xi, \text{grad}_{g}\xi)}} = \frac{ \sum_{j= 1}^{n} g^{\xi j} \frac{\partial}{\partial x^{j}} + g^{\xi \xi} \frac{\partial}{\partial \xi}}{\sqrt{g\left(\sum_{j= 1}^{n} g^{\xi j} \frac{\partial}{\partial x^{j}} + g^{\xi \xi} \frac{\partial}{\partial \xi}, \sum_{i= 1}^{n} g^{\xi i} \frac{\partial}{\partial x^{i}}+ g^{\xi \xi} \frac{\partial}{\partial \xi}\right)}} \\
    & = \frac{1}{\sqrt{g^{\xi \xi}}} \left( \sum_{j= 1}^{n} g^{\xi j} \frac{\partial}{\partial x^{j}} + g^{\xi \xi} \frac{\partial}{\partial \xi} \right).
\end{align*}
It follows (\ref{Angle:eqn1}) can be represented by
\begin{align*}
   \frac{g(\partial_{\xi}, \nu_{g})}{g(\partial_{\xi}, \partial_{\xi})^{\frac{1}{2}}} & =  \frac{g\left(\partial_{\xi}, \frac{1}{\sqrt{g^{\xi \xi}}} \left( \sum_{j= 1}^{n} g^{\xi j} \frac{\partial}{\partial x^{j}} + g^{\xi \xi} \frac{\partial}{\partial \xi} \right)\right)}{g(\partial_{\xi}, \partial_{\xi})^{\frac{1}{2}}} = \frac{1}{\sqrt{g_{\xi \xi} g^{\xi \xi}}} \left( \sum_{j = 1}^{n} g_{\xi j}g^{\xi j} + g_{\xi \xi} g^{\xi \xi} \right) \\
   & = \frac{1}{\sqrt{g_{\xi \xi} g^{\xi \xi}}}  \delta_{\xi}^{\xi} = \frac{1}{\sqrt{g_{\xi \xi} g^{\xi \xi}}}.
\end{align*}
Note that $ g_{\xi\xi} g^{\xi \xi} = g(\partial_{\xi}, \partial_{\xi}) g^{-1}(d\xi, d\xi) $ that only depends on the global variable $ \xi $ and the coordinate free tensors $ g $ and $ g^{-1} $, and so is independent of the choice of local coordinates other than $ \xi $-direction. Hence the angle condition has three equivalent expressions as stated below:
\begin{equation}\label{Angle:eqn2}
\cos^{-1} \left( \frac{g(\partial_{\xi}, \nu_{g})}{g(\partial_{\xi}, \partial_{\xi})^{\frac{1}{2}}} \right) \in [0, \frac{\pi}{4}) \Leftrightarrow \frac{g(\partial_{\xi}, \nu_{g})}{g(\partial_{\xi}, \partial_{\xi})^{\frac{1}{2}}} \in (\frac{\sqrt{2}}{2}, 1] \Leftrightarrow g_{\xi \xi} g^{\xi \xi} \in [1, 2) \; {\rm along} \; X_{0}.
\end{equation}
We are now ready to show that the main technical result (\ref{intro:eqn2}) holds.
\begin{theorem}\label{Angle:thm1}
Let $ X $ be a closed, oriented manifold, $ \dim X \geqslant 2 $. Let $ g $ be a complete metric on $ X \times \R $ that is of bounded curvature. Assume also that the smallest eigenvalue of the symmetric $ 2 $-tensor $ g $ is uniformly bounded below by some positive constant $ \lambda > 0 $. If $ \chi(X) $ = 0, then there exists a global diffeomorphism
\begin{equation*}
F : X \times \R \rightarrow X \times \R
\end{equation*}
with inverse $ G : = F^{-1} $, such that with $ \tilde{\xi} : = \pi_{R} \circ G $,
\begin{equation}\label{Angle:eqn3}
 1 < g\left(\partial_{\tilde{\xi}}, \partial_{\tilde{\xi}}\right)g^{-1}\left(d\tilde{\xi}, d\tilde{\xi}\right) < 2 \; {\rm on} \; \{ \tilde{\xi} = 0 \}.
\end{equation}
\end{theorem}
\begin{proof}
Let $ \partial_{\xi} \in \Gamma(T\R) \hookrightarrow \Gamma(T(X \times \R)) $ be the fixed vector field as above. By assumption the smallest eigenvalue of $ g $ is uniformly bounded below by $ \lambda > 0 $, it follows that there exist $ C_{1}, C_{2} > 0 $ such that both $ g_{\xi \xi} \geqslant C_{1} $ and $ g^{\xi \xi} \leqslant C_{2} $ hold uniformly on $ X \times \R $. Without loss of generality, we may thus assume that 
\begin{equation*}
    g_{\xi \xi} > 1, \sqrt{g^{\xi \xi}} \leqslant 1 - \zeta \; \text{uniformly on} \; X \times \R
\end{equation*} for some $ \zeta \in (0, 1) $ by a one-time scaling of the original metric $ g $ if necessary. 

Since $ \chi(X) = 0 $, there exists a nowhere vanishing global vector field $ \bar{V} \in \Gamma(TX) $. Recall that $ \pi_{X} : X \times \R \rightarrow X $ and $ \pi_{\R} : X \times \R \rightarrow \R $ are natural projections. We define
\begin{equation*}
    V \in \Gamma(T(X \times \R)), T_{p}(X \times \R) \ni V_{p} : = (\bar{V}_{p}, 0) \in T_{\pi_{X}(p)}X \oplus T_{\pi_{\xi}(p)}\R, \forall p \in X \times \R.
\end{equation*}
as a global vector field on $ X \times \R $. $ V $ is a complete, global, nowhere vanishing vector field whose integral curve starting at any $ p \in X \times \R $ is in the fiber $ \{ \xi = \pi_{\xi}(p) \} $ with infinitesimal generator $ \bar{V} $. Clearly, the global nowhere vanishing vector field  $ \partial_{\xi} \in \Gamma(T(X \times \R)) $ and $ V $ are linearly independent. Set
\begin{equation*}
U = \frac{V - g\left(\frac{\partial_{\xi}}{\lVert \partial_{\xi} \rVert_{g}}, V \right) \frac{\partial_{\xi}}{\lVert \partial_{\xi} \rVert_{g}}}{\left\lVert V - g\left(\frac{\partial_{\xi}}{\lVert \partial_{\xi} \rVert_{g}}, V \right) \frac{\partial_{\xi}}{\lVert \partial_{\xi} \rVert_{g}} \right\rVert_{g}}
\end{equation*}
It follows from linear independency of $ V $ and $ \partial_{\xi} $ that $ U \in \Gamma(T(X \times \R)) $ is a global nowhere vanishing unit vector field. In summary,
\begin{equation}\label{Angle:eqn4}
g(U, U) = 1, g(U, \partial_{\xi}) = 0.
\end{equation}
It follows from the positive definiteness of $ g $ that $ U \neq 0 $ everywhere on $ X \times \R $. We also claim that $ U $ and $ \partial_{\xi} $ are linearly independent. By Sylvester's criterion of positive definite matrix $ g $ on each $ T_{p}(X \times \R) $, we show that the $ 2 \times 2 $ principal minor of $ g $, say
\begin{equation*}
    G_{U, \partial_{\xi}} : = \begin{bmatrix} g(U, U) & g(U, \partial_{\xi}) \\ g(U, \partial_{\xi}) & g(\partial_{\xi}, \partial_{\xi}) \end{bmatrix}
\end{equation*}
is invertible. By (\ref{Angle:eqn4}), it is straightforward to check that
\begin{equation*}
\det(G_{U, \partial_{\xi}}) = g_{\xi\xi} > 0.
\end{equation*}
Since $ X $ is a closed manifold, and the projection of $ U $ onto $ X $ is proportional to $ V $, we conclude that $ U $ is again a complete vector field. 

We denote by $ \Phi : \R \times (X \times \R) \rightarrow X \times \R $ the global smooth flow of $ U $. 
We define a map
\begin{equation}\label{Angle:eqn5}
    F : X \times \R \rightarrow X \times \R, F(p) = \Phi_{\pi_{\R}(p)}(p), \forall p \in X \times \R
\end{equation}
We will prove that is a diffeomorphism, with inverse $G:X \times \R \rightarrow X \times \R$ of the form $G(q) = \Phi_{t_q}(q)$, such that for $q = F(p)$, 
\begin{equation}\label{Angle:eqn3b}
    G(F(p)) = p \Leftrightarrow \Phi_{t}(\Phi_{\pi_{\R}(p)}(p)) = p \Leftrightarrow t + \pi_{\R}(p) = 0 \Leftrightarrow \forall q \in X \times \R, \exists! t \; {\rm s.t.} \; t + \pi_{\R}(\Phi_{t}(q)) = 0.
\end{equation}
Equivalently, we need to show that $ t = t_{q} $ in (\ref{Angle:eqn3b}) has a unique solution for each $q$ and is a smooth function of $ q $. 

To verify this, we {\it{Claim}} that for each fixed $ q \in X \times \R $
\begin{equation*}
    h_{q}(t) : = t +  \pi_{\R}(\Phi_{t}(q))
\end{equation*}
is a monotonic function from $ \R $ to $ \R $, and that $ \lim_{t \rightarrow +\infty}h_{q}(t) = + \infty $, and $ \lim_{t \rightarrow -\infty}h_{q}(t) = - \infty $. 

The claim implies that $ h_{q}(t) = 0 $ for a unique $ t \in \R $, and the smoothness in $q$ will follow.

By definition, $ h_{q}(t) $ is smooth in both $ t $ and $ q $. Observe that
\begin{equation*}
    \frac{dh_{q}(t)}{dt} = 1 - \frac{g(\frac{\partial_{\xi}}{\lVert \partial_{\xi} \rVert_{g}}, V)}{\lVert \partial_{\xi} \rVert_{g}\left\lVert V - g\left(\frac{\partial_{\xi}}{\lVert \partial_{\xi} \rVert_{g}}, V \right) \frac{\partial_{\xi}}{\lVert \partial_{\xi} \rVert_{g}} \right\rVert_{g}},
\end{equation*}
the denominator is clearly nonzero.

To estimate the right hand side of this equation, we claim that for any global $ g $-unit vector field $ \Gamma(TX) \in W \Leftrightarrow (W, 0) \in \Gamma(TX \oplus T\R) $, 
\begin{equation}\label{Angle:eqn4a}
    \max_{W \in \Gamma(TX), \lVert W \rVert_{g} = 1} g^{2}\left(\frac{\partial_{\xi}}{\lVert \partial_{\xi} \rVert_{g}}, W\right) \leqslant 1 - \frac{1}{g_{\xi\xi}g^{\xi \xi}}.
\end{equation}
Denote by $ Y = \frac{\partial_{\xi}}{\lVert \partial_{\xi} \rVert_{g}} $, $ \lVert Y \rVert_{g} = 1 $. Note that $ T_{q}(X \times \R) = TX \oplus (TX)^{\perp} $ where $ (TX)^{\perp} $ is the $ g $-normal bundle. It follows that
\begin{equation*}
    Y = Y_{TX} + Y^{\perp}
\end{equation*}
uniquely. By $ g $-orthogonal decomposition, it is immediate that for any $ W \in \Gamma(TX) $,
\begin{equation*}
    g^{2}(Y, W) = g^{2}(Y_{TX}, W) \leqslant \lVert Y_{TX} \rVert_{g}^{2} = 1 - \lVert Y^{\perp} \rVert_{g}^{2} 
\end{equation*}
with equality if and only if $ W = \pm\frac{Y_{TX}}{\lVert Y_{TX} \rVert_{g}} $. By flowout theorem, we can choose local coordinates $ (x^{1}, \dotso, x^{n}, \xi) $ within some small chart around $ q $ such that $ T_{x}X = Span \{\partial_{1}, \dotso, \partial_{n}) $ for $ x $ near $ q $; in addition, the last coordinate is the original $ \xi $-coordinate with the original $ \partial_{\xi} $ coordinate vector field locally. In these coordinates, there exist smooth functions $a^i$ defined near $q$ with
\begin{equation*}
    Y^{\perp} = Y - Y_{TX} = \frac{\partial_{\xi}}{\sqrt{g_{\xi \xi}}} - \sum_{i = 1}^{n} a_{i} \partial_{i}.
\end{equation*}
Note that
\begin{equation}\label{Angle:eqn4b}
d\xi(Y^{\perp}) = g((d\xi)^{\sharp}, Y^{\perp}).
\end{equation}
We compute both sides of (\ref{Angle:eqn4b}). On one hand,
\begin{equation*}
  d\xi(Y^{\perp}) = d\xi\left(\frac{\partial_{\xi}}{\sqrt{g_{\xi \xi}}} - \sum_{i = 1}^{n} a_{i} \partial_{i} \right) = \frac{1}{\sqrt{g_{\xi \xi}}}.
\end{equation*}
On the other hand, locally $ (d\xi)^{\sharp} = \sum_{i = 1}^{n} g^{i\xi} \partial_{i} + g^{\xi \xi} \partial_{\xi} $, therefore
\begin{equation*}
    g((d\xi)^{\sharp}, Y^{\perp}) = g\left(\sum_{i = 1}^{n} g^{i\xi} \partial_{i} + g^{\xi \xi} \partial_{\xi}, Y^{\perp}\right) = g^{\xi \xi} g(\partial_{\xi}, Y^{\perp}).
\end{equation*}
By (\ref{Angle:eqn4b}), we get
\begin{equation}\label{Angle:eqn4c}
d\xi(Y^{\perp}) = g((d\xi)^{\sharp}, Y^{\perp}) \Leftrightarrow \frac{1}{\sqrt{g_{\xi \xi}}} = g^{\xi \xi} g(\partial_{\xi}, Y^{\perp}) \Leftrightarrow g(\partial_{\xi}, Y^{\perp}) = \frac{1}{\sqrt{g_{\xi \xi}} g^{\xi \xi}}.
\end{equation}
Finally by (\ref{Angle:eqn4c}), we have
\begin{align*}
    \lVert Y^{\perp} \rVert_{g}^{2} & = g(Y^{\perp}, Y^{\perp}) = g\left(\frac{\partial_{\xi}}{\sqrt{g_{\xi \xi}}} - \sum_{i = 1}^{n} a_{i} \partial_{i}, Y^{\perp} \right) = \frac{1}{\sqrt{g_{\xi \xi}}} g(\partial_{\xi}, Y^{\perp}) = \frac{1}{g_{\xi \xi}g^{\xi \xi}}.
\end{align*}
Hence (\ref{Angle:eqn4a}) follows. With (\ref{Angle:eqn4a}), we can compute that
\begin{align*}
    & \frac{\left\lvert g(\frac{\partial_{\xi}}{\lVert \partial_{\xi} \rVert_{g}}, V) \right\rvert}{\lVert \partial_{\xi} \rVert_{g}\left\lVert V - g\left(\frac{\partial_{\xi}}{\lVert \partial_{\xi} \rVert_{g}}, V \right) \frac{\partial_{\xi}}{\lVert \partial_{\xi} \rVert_{g}} \right\rVert_{g}} \\
    & \qquad \leqslant \frac{\sqrt{g(V, V)}}{\sqrt{g_{\xi \xi}} \sqrt{g(V, V) - \left( g\left(\frac{\partial_{\xi}}{\lVert \partial_{\xi} \rVert_{g}}, V \right)\right)^{2}}} = \frac{\sqrt{g(V, V)}}{\sqrt{g_{\xi \xi}} \sqrt{g(V, V) - g(V, V)\left( g\left(\frac{\partial_{\xi}}{\lVert \partial_{\xi} \rVert_{g}}, \frac{V}{\lVert V \rVert_{g}} \right)\right)^{2}}} \\
    & \qquad \leqslant \frac{1}{\sqrt{g_{\xi \xi} \cdot \left( 1 - \left(1 - \frac{1}{g_{\xi\xi}g^{\xi\xi}} \right) \right)}} = \sqrt{g^{\xi \xi}}.
\end{align*}
By our assumption $ \sqrt{g_{\xi \xi}} \leqslant 1 - \zeta < 1 $ uniformly on $ X \times \R $. It follows that
\begin{equation}\label{Angle:eqn4d}
\frac{dh_{q}(t)}{dt}= 1 - \frac{g(\frac{\partial_{\xi}}{\lVert \partial_{\xi} \rVert_{g}}, V)}{\lVert \partial_{\xi} \rVert_{g}\left\lVert V - g\left(\frac{\partial_{\xi}}{\lVert \partial_{\xi} \rVert_{g}}, V \right) \frac{\partial_{\xi}}{\lVert \partial_{\xi} \rVert_{g}} \right\rVert_{g}} \geqslant 1 - \sqrt{g^{\xi\xi}} \geqslant \zeta > 0.
\end{equation}
Thus $ h_{q}(t) $ is a monotone increasing function of $ t $ for every $ q \in X \times \R $. By (\ref{Angle:eqn4d}) and fundamental theorem of calculus, $\lim_{t\to \infty} h_q(t) = \infty, \lim_{t\to -\infty} h_q(t) = -\infty$. We therefore conclude that for each $ q \in X \times \R $, there exists a unique $ t \in \R $ such that (\ref{Angle:eqn3b}) holds. Since $ \frac{dh_{t}}{dt} $ is never zero, $ t $ depends smoothly on $ q $ by implicit function theorem. 

This finishes the proof of the {\it{Claim}}. It follows that $ F $ defined in (\ref{Angle:eqn3}) is a global diffeomorphism.
\medskip

Our verification of (\ref{Angle:eqn3}) becomes a pointwise argument by picking special local coordinates. Define the new global coordinate
\begin{equation*}
    \tilde{\xi} : = \pi_{\R} \circ G.
\end{equation*}
It follows that we have the following diagram:
\begin{equation}\label{Angle:diag}
    \begin{tikzcd}
        (\cdot, \xi) \in X \times \R \arrow[rr, "F", shift left= 0.5ex] \arrow[dr, "\pi_{\R}" '] &  & X \times \R \ni (\cdot, \tilde{\xi}) \arrow[ll, "G", shift right= -0.5ex] \arrow[dl, "\pi_{\R} \circ G"] \\ 
        & \R &
    \end{tikzcd}
\end{equation}
Fix an arbitrary point $ Q \in \{ \tilde{\xi} = 0 \} \subset X \times \R_{\xi} $. By flowout theorem, we can choose a local coordinates $ (x^{1}, \dotso, x^{n}, \xi) $ around a given point $ Q $, different from the local coordinates chosen above, such that just at the point $ Q $ centered as $ (0, \dotso, 0, 0) $, we have $ U(Q) = \partial_{1} $; also with this local coordinates, we have $ g_{ij}(Q) = g^{ij}(Q) = \delta_{i}^{j}, i, j \neq 1, \xi $, $ g_{i\xi}(Q) = 0, i \neq 1, \xi $. This choice of local coordinates can be chosen such that the last coordinate is the original global coordinate $ \xi $ with the original vector field $ \partial_{\xi} $ locally around $ Q $. By (\ref{Angle:eqn5}), the local expression of $ F $ in $ (x^{1}, \dotso, x^{n}, \xi) $ is of the form
\begin{equation}\label{Angle:eqn6}
F(x, \xi) = (x^{1} + \xi, x^{2}, \dotso, x^{n}, \xi) + O(\lvert (x, \xi) \rvert^{2}).
\end{equation}
Note that $ \lvert (x, \xi) \rvert \ll 1 $ in this local coordinate expression (\ref{Angle:eqn6}). It follows from (\ref{Angle:eqn6}) that
\begin{equation}\label{Angle:eqn7}
JF_{Q} = \begin{bmatrix} 1 & 0 & \dotso & 0 & 1 \\ 0 & 1 & \dotso & 0 & 0 \\ \vdots & \vdots & \ddots & \vdots & \vdots \\ 0 & 0 & \dotso & 1 & 0 \\ 0 & 0 & \dotso & 0 & 1 \end{bmatrix}, JF_{Q}^{-1} = \begin{bmatrix} 1 & 0 & \dotso & 0 & - 1 \\ 0 & 1 & \dotso & 0 & 0 \\ \vdots & \vdots & \ddots & \vdots & \vdots \\ 0 & 0 & \dotso & 1 & 0 \\ 0 & 0 & \dotso & 0 & 1 \end{bmatrix}.
\end{equation}
By (\ref{Angle:diag}), we see that $ G $ maps $ (\cdot, \tilde{\xi}) $ to $ (\cdot, \xi) $. It follows that $ g(U, \nabla (\pi_{\R} \circ G)) = g(U, \partial_{\xi}) = 0 $. Therefore, $ U \in \Gamma(\{ \tilde{\xi} = 0\}) $. By our choice of local coordinates and the same argument, we show that $ (x^{1}, \dotso, x^{n}) \in \{ \tilde{\xi} = 0 \} $. Analogously, when we choose another collection of local coordinates $ (y^{1}, \dotso, y^{n}, \xi) $ around $ Q' $, where $ Q' $ and $ Q $ are close enough such that two charts have nontrivial intersection, the transitions satisfy $ y^{i} = y^{i}(x) $ only, it follows that $ \partial_{\tilde{\xi}} $ is unchanged among different choices of local coordinates. This verifies our claim that the verification of (\ref{Angle:eqn3}) is a pointwise argument.

We now compute $ g(\partial_{\tilde{\xi}}, \partial_{\tilde{\xi}})(Q) g^{-1}(d\tilde{\xi}, d\tilde{\xi})(Q) $ at any point $ Q \in \{ \tilde{\xi} = 0 \}\subset X \times \R $ with the local coordinates chosen above. Just at $ Q $, we apply the local expression of F in (\ref{Angle:eqn6}) that
\begin{align*}
g(\partial_{\tilde{\xi}}, \partial_{\tilde{\xi}})(Q) g^{-1}(d\tilde{\xi}, d\tilde{\xi})(Q) & = g(-\partial_{x^{1}} + \partial_{\xi}, -\partial_{x^{1}} + \partial_{\xi}) g^{-1}(d\xi, d\xi) = (g_{11} - 2g_{1\xi} + g_{\xi \xi}) g^{\xi \xi} \\
& = (g_{11} - 2g_{1\xi} + g_{\xi \xi}) \cdot \frac{g_{11}}{g_{11}g_{\xi \xi} - g_{1 \xi}^{2}}.
\end{align*}
With $ U(Q) = \partial_{1} $, it follows from (\ref{Angle:eqn4}) that
\begin{equation*}
g(\partial_{\tilde{\xi}}, \partial_{\tilde{\xi}})(Q) g^{-1}(d\tilde{\xi}, d\tilde{\xi})(Q) = (g_{11} - 2g_{1\xi} + g_{\xi \xi}) \cdot \frac{g_{11}}{g_{11}g_{\xi \xi} - g_{1 \xi}^{2}} = \frac{1}{g_{\xi\xi}} + 1.
 \end{equation*}
Our calculation is consistent with the general fact that $ g(\partial_{\tilde{\xi}}, \partial_{\tilde{\xi}}) g^{-1}(d\tilde{\xi}, d\tilde{\xi}) \geqslant 1 $. Recall that $ g_{\xi \xi} > 1 $. It follows that (\ref{Angle:eqn3}) holds at $ Q $. Repeating the same argument at each point of $ \{ \tilde{\xi} = 0 \} $ with associated choices of local coordinates, we conclude that (\ref{Angle:eqn3}) holds on the hypersurface $ \{ \tilde{\xi} = 0 \} $.
\end{proof}
\begin{remark}\label{Angle:re1}
By the same argument in the derivation of (\ref{Angle:eqn2}), we have
\begin{equation}\label{Angle:eqn8}
\cos^{-1} \left( \frac{g(\partial_{\tilde{\xi}}, \nu_{g, \tilde{\xi}})}{g(\partial_{\tilde{\xi}}, \partial_{\tilde{\xi}})^{\frac{1}{2}}} \right) \in [0, \frac{\pi}{4}) \Leftrightarrow \frac{g(\partial_{\tilde{\xi}}, \nu_{g, \tilde{\xi}})}{g(\partial_{\tilde{\xi}}, \partial_{\tilde{\xi}})^{\frac{1}{2}}} \in (\frac{\sqrt{2}}{2}, 1] \Leftrightarrow g(\partial_{\tilde{\xi}}, \partial_{\tilde{\xi}}) g^{-1}(d\tilde{\xi}, d\tilde{\xi}) \in [1, 2) \; {\rm on} \; \{ \tilde{\xi} = 0 \}
\end{equation}
where $ \nu_{g, \tilde{\xi}} $ is the unit normal vector field along $ \{ \tilde{\xi} = 0 \} $.
\end{remark}
\medskip

\section{Generalization of The $ \mathbb{S}^{1} $-Stability Conjecture}
We have shown in \S2 that (\ref{Angle:eqn8}) holds. By Remark \ref{intro:re1}, Theorem \ref{intro:thm2} and a topological result \cite[Proposition 6.4]{Rade}, this geometric angle condition gives a partial answer of the Rosenberg-Stolz conjecture, and a complete answer of the $ \mathbb{S}^{1} $-stability conjecture when $ \dim X \geqslant 5 $, provided that $ \chi(X) = 0 $. We also generalize the $ \mathbb{S}^{1} $-stability conjecture to a $ \mathbb{T}^{n} $-stability theorem with $ \chi(X) = 0 $ and $ \dim X \geqslant 5 $. The main analysis tools are the key results of \cite{RX2} and \cite{XU11} with the geometric angle condition, the analysis of geometric elliptic PDEs and conformal geometry with the introduction of auxiliary spaces.
In particular, the results in this section hold for all odd dimensional, oriented, closed manifolds $ X $ with dimension at least $ 5 $.

Let $ a \in \R^{+} $ be some constant. We say that a closed embedded hypersurface $ \Sigma $ {\it{separates}} the two disjoint components $ \partial_{a, -} Y : = X \times \{ -a \} $ and $ \partial_{a, +} Y : = X \times \{ a \} $ of a compact cylinder $ Y = X \times [-a, a] $ if no connected component of $ Y \backslash \Sigma $ admits a path $ \gamma : [0, 1] \rightarrow Y $ such that $ \gamma(0) \in \partial_{a, -} Y $ and $ \gamma(1) \in \partial_{a, +} Y $ \cite[Definition 2.1]{Rade}. We will need the following topological result, which holds when $ \dim X \geqslant 5 $.
\begin{proposition}\cite[Propositon 6.4]{Rade}\label{RS:prop1}
Let $ X $ be a closed, connected, oriented manifold with $ \dim X \geqslant 5 $. Let $ Y = X \times [-a, a]_{\xi} $ for some $ a \in \R^{+} $. Let $ \Sigma $ be a closed, connected, oriented hypersurface separating $ \partial_{a, -} Y $ and $ \partial_{a, +} Y $. If $ \Sigma $ admits a PSC metric, then so does $ X $. 
\end{proposition}
We first address the Rosenberg-Stolz conjecture that $X\times \R$ admits a complete PSC metric if and only if $X$ admits a PSC metric.
\begin{theorem}\label{RS:thm1}
Let $ X $ be a closed, oriented manifold, $ \dim X \geqslant 5 $. Let $ g $ be a complete metric on $ X \times \R $ that is of bounded curvature. Assume also that the smallest eigenvalue of $ g $ is uniformly bounded below by $ \lambda > 0 $. If $ \chi(X) = 0 $ and $ R_{g} \geqslant \kappa > 0 $ on $ X \times \R $ uniformly for some constant $ \kappa $, then $ X $ admits a PSC metric $ \bar{g} $.
\end{theorem}
\begin{proof}
By Theorem \ref{Angle:thm1} and (\ref{Angle:eqn8}), there exists a global diffeomorphism $ F : X \times \R \rightarrow X \times \R $ and $ \tilde{\xi} = \pi_{\R} \circ G $ such that
\begin{equation*}
    \angle_{g}(\nu_{\tilde{g}}, \partial_{\tilde{\xi}}) \in [0, \frac{\pi}{4}) \; {\rm on} \; \{ \tilde{\xi} = 0 \},
\end{equation*}
By Theorem \ref{intro:thm2} and Remark \ref{intro:re1}, there exists a metric $ \bar{g} $ conformal to the metric $ \tilde{g} $, such that $ R_{\imath_{0}^{*}\bar{g}} > 0 $ on $ {\tilde{\xi} = 0} $, where $ \imath_{0} : \{ \tilde{\xi} = 0 \} \rightarrow X \times \R $.

Recall that in the proof of Theorem \ref{Angle:thm1}, we know by (\ref{Angle:eqn3b})that for each $ q \in X \times \R $ there exists a unique $ t $ such that $ t(q) = -\pi_{\R}(\Phi_{t}(q)) = -\pi_{R}(p) $ where $ p = G(q) $, which follows that $ t = -\pi_{\R} \circ G = -\tilde{\xi} $. By the limiting behavior of $ h_{t}(q) $ we verified in Theorem \ref{Angle:thm1}, it follows immediately that $ \{ \tilde{\xi} = 0 \} \subset Y = X \times [-a, a]_{\xi} $ for some large enough $ a $. In addition, $ \{ \tilde{\xi} = 0 \} $ separates $ \partial Y_{a, -} $ and $ \partial Y_{a, +} $.

Since $ \{ \tilde{\xi} = 0 \} $ admits a PSC metric and $ \dim X \geqslant 5 $, we conclude by Proposition \ref{RS:prop1} that $ X $ admits a PSC metric.
\end{proof}

The above result immediately gives the positive answer of the $ \mathbb{S}^{1} $-stability conjecture with the same hypothesis on $ X $.
\begin{corollary}\label{RS:cor1}
Let $ X $ be a closed, oriented manifold, $ \dim X \geqslant 5 $, with $\chi(M) = 0$. 
Then  $ X \times \mathbb{S}^{1} $ admits a PSC metric if and only if $ X $ admits a PSC metric.
\end{corollary}
\begin{proof}
For the nontrivial direction, let $ \Pi : X \times \R \rightarrow X \times \mathbb{S}^{1} $ be the smooth covering map. Since $ X \times \mathbb{S}^{1} $ admits a PSC metric $ g $, it is immediate that $ \Pi^{*}g $ is a complete metric that is of bounded curvature, $ R_{\Pi^{*}g} $ is uniformly bounded below by some positive constant $ \kappa $, and such that the smallest eigenvalue of $ \Pi^{*}g $ is uniformly bounded below by some positive constant $ \lambda > 0 $. The rest of the proof follows from Theorem \ref{RS:thm1}.
\end{proof}
Note that $ X \times (\mathbb{S}^{1})^{n} $ is a special manifold that has zero characteristic, the following lemma shows that for any manifold $ X $, the compact manifolds $ X \times (\mathbb{S}^{1})^{k}, 1 \leqslant k \leqslant n $ admit PSC metrics if and only if $ X \times (\mathbb{S}^{1})^{n} $ admits a PSC metric. 
\begin{lemma}\label{RS:lemma1}
Let $ X $ be a closed, oriented manifold, $ \dim X \geqslant 4 $. For $n>1 $, $ X \times \mathbb{T}^{n} $ admits a PSC metric if and only if $ X \times \mathbb{T}^{n-1} $ admits a PSC metric.
\end{lemma}
\begin{proof}
Identify $ X \times \mathbb{T}^{n} = X \times (\mathbb{S}^{1})^{n} = X \times (\mathbb{S}^{1})^{n - 1} \times \mathbb{S}^{1} $. It is clear that $ \chi(X \times (\mathbb{S}^{1})^{n - 1}) = \chi(X) \cdot (\chi(\mathbb{S}^{1}))^{n - 1}  = 0 $. It follows from Corollary \ref{RS:cor1} that $ X \times (\mathbb{S}^{1})^{n - 1} $ admits a PSC metric.
\end{proof}
Lemma \ref{RS:lemma1} serves as a recursion result in terms of the power of $ \mathbb{S}^{1} $. This allows us to generalize the $ \mathbb{S}^{1} $-stability conjecture to a {\it{$ \mathbb{T}^{n} $-stability theorem}}, provided that $ \chi(X) = 0 $. 
\begin{corollary}\label{RS:cor2}
Let $ X $ be a closed, oriented manifold, $ \dim X \geqslant 5 $, with $ \chi(X) = 0 $. Then $ X \times \mathbb{T}^{n} $ admits a PSC metric if and only if $ X $ admits a PSC metric.
\end{corollary}
\begin{proof}
By Lemma \ref{RS:lemma1}, we know that $ X \times \mathbb{S}^{1} $ admits a PSC metric. Since $ \chi(X) = 0 $ and $ \dim X \geqslant 5 $, Corollary \ref{RS:cor1} applies, and we conclude that $ X $ admits a PSC metric.
\end{proof}

\bibliographystyle{plain}
\bibliography{ScalarPre}

\begin{thebibliography}{10}

\bibitem{Aubin}
T.~Aubin.
\newblock {\em Nonlinear Analysis on Manifolds. {M}onge-{A}mp\'ere
  {E}quations.}
\newblock Grundlehren der mathematischen Wissenschaften. Springer, Berlin,
  Heidelberg, New York, 1982.

\bibitem{CRZ}
S.~Cecchini, D.~R\"ade, and R.~Zeidler.
\newblock Nonnegative scalar curvature on manifolds with at least two ends.
\newblock {\em J. Topol.}, 16:855--876, 2023.

\bibitem{CZ}
S.~Cecchini and R.~Zeidler.
\newblock Scalar and mean curvature comparison via the {D}irac operator.
\newblock {\em Geometry and Topology}, 28:1167--1212, 2024.

\bibitem{Chodosh}
O.~Chodosh.
\newblock Stable minimal surfaces and positive scalar curvature.
  \url{https://web.stanford.edu/~ochodosh/Math258-min-surf.pdf}.

\bibitem{CL}
O.~Chodosh and C.~Li.
\newblock Generalized soap bubbles and the topology of manifolds with positive
  scalar curvature.
\newblock {\em Ann. of Math. (2)}, 199:707--740, 2024.

\bibitem{CMSW}
O.~Chodosh, C.~Mantoulidis, F.~Schulze, and Z.~Wang.
\newblock Generic regularity for minimizing hypersurfaces in dimension $ 11 $.
\newblock {\em ar{X}iv:2506.12052}.

\bibitem{GROMOV}
M.~Gromov.
\newblock Four lectures on scalar curvature.
\newblock {\em ar{X}iv:1908.10612v6}.

\bibitem{GROMOV2}
M.~Gromov.
\newblock Metric inequalities with scalar curvature.
\newblock {\em Geom. Funct. Anal.}, 28(3):645--726, 2018.

\bibitem{GromovL}
M.~Gromov and H.~Lawson.
\newblock Positive scalar curvature and the {D}irac operator on complete
  riemannian manifolds.
\newblock {\em Publ. Math. {I}H\'ES}, 58:295--408, 1983.

\bibitem{LM}
L.~Markus.
\newblock Line element fields and {L}orentz structures on differentiable
  manifolds.
\newblock {\em Ann. Math.}, 62(3):411 -- 417, 1955.

\bibitem{Rade}
D.~R\"ade.
\newblock Scalar and mean curvature comparison via $ \mu $-bubbles.
\newblock {\em Calc. Var. Partial Differential Equations}, 62(187), 2023.

\bibitem{JR}
J.~Rosenberg.
\newblock Manifolds of positive scalar curvature: a progress report.
\newblock {\em Surveys in Differential Geometry}, 11(1):259--294, 2006.

\bibitem{RosSto}
J.~Rosenberg and S.~Stolz.
\newblock Manifolds of positive scalar curvature.
\newblock {\em Algebraic topology and its applications. Vol. 27. Math. Sci.
  Res. Inst. Publ}, 27:241--267, 1994.

\bibitem{RX2}
S.~Rosenberg and J.~Xu.
\newblock A codimension two approach to the $ \mathbb{S}^{1} $-stability
  conjecture.
\newblock {\em ar{X}iv:2412.12479}.

\bibitem{SY}
R.~Schoen and S.-T. Yau.
\newblock Conformally flat manifolds, {K}leinian groups and scalar curvature.
\newblock {\em Invent. Math.}, 92:47--71, 1988.

\bibitem{XU11}
J.~Xu.
\newblock Existence of positive scalar curvature and positive yamabe constant
  on hypersurfaces of noncompact cylinders.
\newblock {\em ar{X}iv:2509.24016}.

\bibitem{Xu12}
J.~Xu.
\newblock On the {R}osenberg-{S}tolz conjecture for $ {X} \times \mathbb{R}^{2}
  $ and its application in complex geometry.
\newblock {\em ar{X}iv:2510.13588}.

\bibitem{XU10}
J.~Xu.
\newblock Scalar and mean curvature comparison on compact cylinder.
\newblock {\em ar{X}iv:2507.07005}.

\bibitem{Zeidler2}
R.~Zeidler.
\newblock An index obstruction to positive scalar curvature on fiber bundles
  over aspherical manifolds.
\newblock {\em Algebr. Geom. Topol.}, 17:3081--3094, 2017.

\bibitem{Zeidler}
R.~Zeidler.
\newblock Band width estimates via the {D}irac operator.
\newblock {\em J. Differential Geom.}, 122(1):155--183, 2022.

\end{thebibliography}
\end{document}